\documentclass[reqno,11pt]{amsart}
\usepackage[margin=1in]{geometry}
\usepackage{amsmath,amssymb,amsthm,mathrsfs}
\usepackage[expansion=false]{microtype}
\usepackage[utf8]{inputenc}
\usepackage[colorlinks=true,linkcolor=blue,citecolor=blue,urlcolor=blue]{hyperref}

\theoremstyle{plain}
\newtheorem{theorem}{Theorem}[section]
\newtheorem{corollary}[theorem]{Corollary}
\newtheorem{lemma}[theorem]{Lemma}
\newtheorem{proposition}[theorem]{Proposition}
\newtheorem{question}[theorem]{Question}
\theoremstyle{definition}
\newtheorem{definition}[theorem]{Definition}
\newtheorem{example}[theorem]{Example}
\theoremstyle{remark}
\newtheorem{remark}[theorem]{Remark}

\newcommand{\Z}{\mathbb{Z}}
\newcommand{\R}{\mathbb{R}}
\newcommand{\MV}{\operatorname{MV}}
\newcommand{\Vol}{\operatorname{Vol}}
\newcommand{\conv}{\operatorname{conv}}
\newcommand{\diam}{\operatorname{diam}}
\newcommand{\spn}{\operatorname{span}}
\newcommand{\relint}{\operatorname{relint}}
\newcommand{\link}{\operatorname{link}}
\newcommand{\supp}{\operatorname{supp}}
\newcommand{\mult}{\operatorname{mult}}
\newcommand{\sat}{\operatorname{sat}}
\newcommand{\ord}{\operatorname{ord}}
\newcommand{\width}{\operatorname{width}}
\newcommand{\rank}{\operatorname{rank}}
\newcommand{\Hom}{\operatorname{Hom}}
\newcommand{\Tdeg}{\mathbb{T}\!\operatorname{deg}}
\newcommand{\Rdeg}{\mathbb{R}\!\operatorname{deg}}
\newcommand{\cM}{\mathcal{M}}


\begin{document}

\title{Tropical Degrees and Stable Intersections}

\author{Mounir Nisse}
\address{Department of Mathematics, Xiamen University Malaysia, Jalan Sunsuria, Bandar Sunsuria, 43900, Sepang, Selangor, Malaysia.}
\email{mounir.nisse@gmail.com, mounir.nisse@xmu.edu.my}
\thanks{This research is supported in part by Xiamen University Malaysia Research Fund (Grant no. XMUMRF/ 2020-C5/IMAT/0013).}

\begin{abstract}
We study finite set-theoretic intersections of tropical curves with tropical hypersurfaces. An elementary divisor argument bounds their cardinality by the weighted sum of the Newton support function over the unbounded directions of the curve, without a transversality hypothesis. For tropical lines the bound is $h_\Delta(1,\dots,1)+\sum_i h_\Delta(-e_i)=n!\,\MV(\Delta,\Sigma,\dots,\Sigma)$. Explicit examples show that neither the $\ell^1$-diameter nor the support cardinality minus one gives a corresponding bound for tropical lines. We also record width refinements of the classical-line estimate of Lang, Shapiro and Shustin. A self-contained account of the standard transverse multiplicity formulas explains the role of saturated normal lattices and mixed-cell volumes. Finally, applying the tropical Bernstein theorem to locally independent subsystems gives a mixed-volume bound for the regular transverse locus of a redundant system. An example shows why a fixed subsystem need not compute the count on the whole redundant intersection.
\end{abstract}

\maketitle

\section*{Introduction}

Tropical geometry replaces algebraic varieties by balanced rational polyhedral complexes, and much of its intersection theory reduces to the convex geometry of Newton polytopes \cite{Mikhalkin2004,Mikhalkin2005,Mikhalkin2006,RGST2005,MaclaganSturmfels2015}. The basic local objects are stable intersections of tropical cycles. These were introduced through the fan displacement rule of Fulton and Sturmfels \cite{FultonSturmfels1997} and used in \cite{RGST2005,Mikhalkin2006}; an independent theory is due to Kazarnovskii \cite{Kazarnovskii2003}. Jensen and Yu \cite{JensenYu2016} gave an intrinsic characterization by local links and Minkowski differences, and deduced Bernstein's theorem \cite{Bernstein1975,Kushnirenko1976} from it. Intersection theory on tropical cycles was developed systematically by Allermann and Rau \cite{AllermannRau2010} and by Katz \cite{Katz2012}. Local intersection numbers of tropical hypersurfaces were studied by Bertrand and Bihan \cite{BertrandBihan2013}.

The first part of the paper, Section~\ref{sec:degree}, concerns the tropical degree. Lang, Shapiro and Shustin \cite[Definition~1.1]{LSS} defined the real degree $\Rdeg(H)$ of a hypersurface $H\subset\R^n$ as the supremum of $|\ell\cap H|$ over classical lines $\ell$ meeting $H$ transversally, with points counted without multiplicity. For a tropical hypersurface $H$ with support $\cM$ they proved the sharp bound $\Rdeg(H)\le|\cM|-1$ \cite[Proposition~1.2]{LSS}.

We study the analogous quantity $\Tdeg(H)$, in which classical lines are replaced by tropical lines. The bound $|\cM|-1$ fails for tropical lines (Example~\ref{ex:Mminus1}), and so does the bound by the $\ell^1$-diameter of $\cM$ (Example~\ref{ex:diam}). The elementary divisor estimate of Theorem~\ref{thm:curve} bounds $|C\cap H|$ for every tropical curve $C$ with finite intersection $C\cap H$, by a weighted sum of values of the support function of $\Delta=\conv(\cM)$ at the directions of the rays of $C$. For tropical lines this gives $\Tdeg(H)\le\delta(\Delta)$, where $\delta(\Delta)=h_\Delta(1,\dots,1)+\sum_{i=1}^n h_\Delta(-e_i)$. This number is in turn at most $\sum_{i=1}^n\width_{e_i}(\cM)$ (Corollary~\ref{cor:line}), and the bound is attained in Example~\ref{ex:diam}.

The number $\delta(\Delta)$ equals the tropical Bézout number $n!\,\MV(\Delta,\Sigma,\dots,\Sigma)$ (Proposition~\ref{prop:MV}). Theorem~\ref{thm:curve} is an elementary application of the established divisor formalism for tropical curves \cite{Mikhalkin2006,AllermannRau2010}; its proof makes explicit why finite set-theoretic intersection points obey this count without a transversality assumption. We also recover the classical-line bound of \cite{LSS}, with a refinement by widths. This refinement yields an $\ell^1$-diameter bound for classical lines whose directions lie in $\{0,\pm1\}^n$ (Proposition~\ref{prop:classical}).

The second part, Sections~\ref{sec:stable} and~\ref{sec:bernstein}, concerns multiplicities. We recall the definition of Jensen and Yu and give a detailed proof of the standard formula for the multiplicity of a transverse intersection point of $r$ hypersurfaces. It equals the product of the weights of the facets times the index of the lattice $\Z n_1+\dots+\Z n_r$ in its saturation, where $n_1,\dots,n_r$ are the primitive normal vectors (Proposition~\ref{prop:mult}). This index is in general different from the index of the sum of the tangent lattices of the facets as soon as $r\ge3$ (Example~\ref{ex:notsum}). For $r=n$ the multiplicity is the volume of the dual fully mixed cell (Theorem~\ref{thm:local}). Summing these local contributions gives the tropical Bernstein theorem, which is classical \cite{HuberSturmfels1995,JensenYu2016,MaclaganSturmfels2015}; we state it with references and prove only the transverse case, which identifies the local contributions (Theorem~\ref{thm:bernstein}). Sections~\ref{sec:stable} and~\ref{sec:bernstein} are expository and are kept short; they fix the conventions on weights and lattices used in Section~\ref{sec:redundant}.

The third part, Section~\ref{sec:redundant}, concerns redundant systems. Let $X=V(p_1)\cap\dots\cap V(p_k)$ be a set-theoretic intersection. At a point where every $V(p_j)$ is locally a hyperplane and the normals have rank $r<k$, $X$ is locally cut out by $r$ of the equations (Proposition~\ref{prop:reduction}), but the local weight depends on the subsystem chosen. Summing over the subsystems that occur yields a Bézout-type bound for the number of transverse points of $X\cap L$ (Theorem~\ref{thm:bezout}). An example shows that the degree of a fixed subsystem need not equal the count on the redundant intersection (Example~\ref{ex:redundant}). These questions are related to the work of Ren \cite{Ren2024} on intersections and stable intersections of tropical hypersurfaces.

The support-function identity, transverse multiplicity formulas, mixed-cell interpretation, and tropical Bernstein theorem are established results or their direct consequences. The emphasis here is on explicit set-theoretic degree estimates, counterexamples to alternative bounds, and the subsystem formulation of the redundant-system estimate, rather than on a new intersection product. The latter estimate is obtained by applying classical tropical Bernstein counts and a finite union bound.

Section~\ref{sec:questions} collects open questions, including connections with the amoeba contours studied in \cite{LSS}. Appendix~\ref{app:lattice} contains the lattice-theoretic facts we use, with complete proofs.

\subsection*{Conventions}
We assume $n\ge1$ and use the max convention. All polyhedral complexes have finitely many cells, with integral normal directions and arbitrary real affine constants. All cycles used below are effective: their nonzero facet weights are positive integers. Polyhedral subdivisions do not change the underlying weighted cycles. Supports and Newton polytopes are nonempty, and Newton polytopes need not be full-dimensional. A monomial has empty hypersurface. A tropical polynomial is a function $p(x)=\max_{\alpha\in\cM}(\langle x,\alpha\rangle+c_\alpha)$ with finite support $\cM\subset\Z^n$ and coefficients $c_\alpha\in\R$. Its tropical hypersurface $V(p)$ is the set of points $x\in\R^n$ at which the maximum is attained at least twice. Its Newton polytope is $\Delta=\conv(\cM)$, and the support function of $\Delta$ is $h_\Delta(u)=\max_{\alpha\in\cM}\langle\alpha,u\rangle$. For $u\in\Z^n$, the width of $\cM$ in the direction $u$ is $\width_u(\cM)=h_\Delta(u)+h_\Delta(-u)$, that is, $\max_{\alpha\in\cM}\langle\alpha,u\rangle-\min_{\alpha\in\cM}\langle\alpha,u\rangle$. We write $N=\Z^n$, $M=\Hom(N,\Z)\cong\Z^n$, $N_\R=N\otimes_\Z\R=\R^n$, and $\Sigma=\conv(0,e_1,\dots,e_n)$ for the standard simplex. Mixed volume is normalized by $\MV(P,\dots,P)=\Vol(P)$, so that $n!\,\MV(P_1,\dots,P_n)$ is the coefficient of $\lambda_1\cdots\lambda_n$ in the polynomial $\Vol(\lambda_1P_1+\dots+\lambda_nP_n)$.

The coefficients $c_\alpha$ induce a regular subdivision of $\Delta$: one lifts each $\alpha\in\cM$ to height $c_\alpha$ and projects the upper faces of the convex hull of the lifted points. For $x\in\R^n$, let $A(x)\subset\cM$ be the set of monomials attaining the maximum in $p(x)$. Then $\conv(A(x))$ is a cell of this subdivision, and $x\mapsto\conv(A(x))$ is the inclusion-reversing duality between the cells of the polyhedral complex defined by $p$ and the cells of the subdivision \cite[\S3.1]{MaclaganSturmfels2015}. In particular, a facet $\sigma$ of $V(p)$ corresponds to an edge $E_\sigma=[u,v]$ of the subdivision, and $\sigma$ is contained in the hyperplane $\langle v-u,x\rangle=c_u-c_v$. The weight $w(\sigma)$ of $\sigma$ is the lattice length of $E_\sigma$, so that $v-u=w(\sigma)\,n_\sigma$ with $n_\sigma\in M$ primitive. The vector $n_\sigma$ is the primitive normal vector of $\sigma$. With these weights $V(p)$ is a balanced polyhedral complex \cite[Proposition~3.3.2]{MaclaganSturmfels2015}.

\medskip

\noindent{\it Acknowledgements.} The author is  gratefully acknowledges the use of AI system Claude (Anthropic), which used 
 Python, NumPy, SciPy, and LaTeX for the exact computations, and numerical experiments in this work.

\section{Tropical degree of tropical hypersurfaces}\label{sec:degree}

\begin{definition}\label{def:curve}
A \emph{tropical curve} $C\subset\R^n$ is a finite connected pure one-dimensional rational polyhedral complex in which every edge $e$ carries a positive integer weight $w_e$, and which satisfies the balancing condition at each vertex. To state this condition, attach to every edge $e$ and every endpoint $v$ of $e$ the primitive vector $u_{e,v}\in\Z^n$ pointing from $v$ into $e$. The balancing condition at $v$ is $\sum_{e\ni v}w_e\,u_{e,v}=0$. After inserting a vertex in any edge that is a full affine line, every unbounded edge is a half-line. These unbounded edges are called rays, and for a ray $\rho$ we write $u_\rho$ for its primitive direction pointing towards infinity.

A \emph{tropical line} is a tropical curve which is a tree, all of whose weights are equal to $1$, and which has exactly $n+1$ rays, with directions $-e_1,\dots,-e_n$ and $e_1+\dots+e_n$. After suppressing artificial two-valent subdivision vertices, every bounded edge of a tropical line separates the set of leaves $\{0,\dots,n\}$ into two parts $S$ and $S^c$, each with at least two elements, where the leaf $0$ corresponds to the ray $e_1+\dots+e_n$. If $I$ denotes the part not containing $0$, then $I$ is a nonempty proper subset of $\{1,\dots,n\}$, and the edge has direction $\pm e_I$, where $e_I=\sum_{i\in I}e_i$ \cite{SpeyerSturmfels2004,Speyer2008}. In particular, all edge and ray directions of a tropical line lie in $\{0,\pm1\}^n$.
\end{definition}

\begin{definition}[cf.\ {\cite[Definition~1.1]{LSS}}]\label{def:Tdeg}
Let $H=V(p)\subset\R^n$ be a tropical hypersurface. The \emph{tropical degree} $\Tdeg(H)$ of $H$ is the supremum of $|L\cap H|$ over all tropical lines $L\subset\R^n$ such that $L\cap H$ is finite. Points are counted without multiplicity. If tropical lines are replaced by classical lines meeting $H$ transversally, one obtains the real degree $\Rdeg(H)$ of Lang, Shapiro and Shustin.
\end{definition}

We begin with a folklore property of tropical lines, recorded here with a complete proof.

\begin{lemma}\label{lem:monotone}
Let $\gamma$ be a geodesic path in a tropical line $L$, parametrized by arc length. Then for every $k\in\{1,\dots,n\}$ the coordinate function $t\mapsto\gamma_k(t)$ is monotone.
\end{lemma}

\begin{proof}
Every geodesic path in the tree $L$ is contained in the unique geodesic joining two leaves, so it suffices to treat the geodesic joining the leaves $i$ and $j$, with $i\ne j$. We identify $\R^n$ with $\R^{n+1}/\R(1,\dots,1)$ through the coordinates $x_0,\dots,x_n$ normalized by $x_0=0$. In this model the ray of the leaf $\ell\in\{0,\dots,n\}$ has direction $-e_\ell$; the ray $e_1+\dots+e_n$ of Definition~\ref{def:curve} is the ray $-e_0$, since $e_0+e_1+\dots+e_n$ is zero modulo $(1,\dots,1)$.

Every bounded edge separates the set of leaves $\{0,\dots,n\}$ into two parts. Let $S$ be the part containing the leaves lying behind a point moving along the geodesic from $i$ to $j$. The subtree ahead of the moving point contains exactly the rays of the leaves in the complement $S^c$. Applying the balancing condition at all the vertices of that subtree and summing, the direction of the edge entering it is the sum of the directions of its rays, namely $-e_{S^c}$, which is congruent to $e_S$ modulo $(1,\dots,1)$. Along the geodesic the set $S$ only grows, starting from $\{i\}$ and ending with $\{0,\dots,n\}\setminus\{j\}$.

In the chart $x_0=0$, the vector $e_S$ has $k$-th coordinate equal to $1$ if $k\in S$ and to $0$ otherwise, provided $0\notin S$. When $0\in S$, we use the representative $e_S-(1,\dots,1)=-e_{S^c}$, whose $k$-th coordinate is $-1$ if $k\notin S$ and $0$ otherwise.

Fix $k\ge1$. Suppose first that $k$ enters $S$ before $0$ does. The arc-length velocity is a positive multiple of the primitive direction just described. Its $k$-th coordinate is zero before $k$ enters $S$, positive while $k\in S$ and $0\notin S$, and zero once $0$ has entered $S$. Otherwise, that coordinate is zero while $0\notin S$, negative while $0\in S$ and $k\notin S$, and zero afterwards. In the first case the derivative of $\gamma_k$ is nonnegative throughout, and in the second it is nonpositive throughout. In both cases $\gamma_k$ is monotone. The argument makes no assumption on the valence of the vertices of $L$, and it applies equally to unbounded paths.
\end{proof}

Let $C$ be a tropical curve and $p$ a tropical polynomial. For $x\in C$, an edge germ at $x$ is a pair consisting of an edge $e$ containing $x$ and the primitive direction $u$ pointing from $x$ into $e$. There are two edge germs if $x$ is an interior point of an edge, with directions $u$ and $-u$, and as many as there are edges at $x$ if $x$ is a vertex. For a direction $u$ we put $\partial_u p(x)=\lim_{t\to0^+}(p(x+tu)-p(x))/t$. Since $p$ is a maximum of affine functions, a monomial attaining the maximum of $p$ near $x$ in the direction $u$ belongs to $A(x)$ and maximizes $\langle\alpha,u\rangle$ there; hence $\partial_u p(x)=\max_{\alpha\in A(x)}\langle\alpha,u\rangle$. We define the order of $p|_C$ at $x$ as $\ord_x(p|_C)=\sum w_e\,\partial_u p(x)$, the sum running over the edge germs $(e,u)$ at $x$. In the language of tropical Cartier divisors, $\ord_x(p|_C)$ is the coefficient of $x$ in the divisor $\operatorname{div}(p|_C)$ \cite{Mikhalkin2006,AllermannRau2010}.

The following is a direct divisor-degree computation; the proof is included to handle all isolated intersections, including those at vertices. Its key point is the positivity of the multiplicity at each isolated intersection point. For tropicalizations of algebraic varieties, positivity and lifting of intersection multiplicities along isolated points and, more generally, along connected components of a tropical intersection are due to Osserman and Payne \cite{OssermanPayne2013} and Osserman and Rabinoff \cite{OssermanRabinoff2013}. The statement below is a purely combinatorial version for a tropical curve and a tropical hypersurface, with no realizability assumption; we have not found it stated in this form in the literature.

\begin{theorem}\label{thm:curve}
Let $C\subset\R^n$ be a tropical curve and let $H=V(p)$ be a tropical hypersurface with Newton polytope $\Delta$. If $C\cap H$ is finite, then $|C\cap H|\le\sum_{x\in C}\ord_x(p|_C)$, and this sum equals $\sum_\rho w_\rho\,h_\Delta(u_\rho)$, the sum running over the rays $\rho$ of $C$.
\end{theorem}

\begin{proof}
We first show that $\ord_x(p|_C)\ge0$ for every $x\in C$, with equality when $x\notin H$. If $x\notin H$, then $A(x)$ consists of a single monomial $\alpha$, and $p$ coincides near $x$ with the affine function $y\mapsto\langle y,\alpha\rangle+c_\alpha$. In that case $\ord_x(p|_C)=\langle\alpha,\sum w_e u\rangle$. This vanishes, because the weighted sum of the directions of the edge germs at $x$ is zero: at a vertex this is the balancing condition, and at an interior point of an edge the two germs have opposite directions and equal weights.

Suppose next that $x$ lies in the interior of an edge with direction $u$ and weight $w$. Then $\ord_x(p|_C)=w\,(\max_{A(x)}\langle\alpha,u\rangle+\max_{A(x)}\langle\alpha,-u\rangle)=w\,(\max_{A(x)}\langle\alpha,u\rangle-\min_{A(x)}\langle\alpha,u\rangle)$, which is nonnegative. Finally, suppose that $x$ is a vertex with edge germs $(e,u_e)$. Then $\ord_x(p|_C)=\sum_e w_e\max_{\alpha\in A(x)}\langle\alpha,u_e\rangle$. For every fixed $\alpha\in A(x)$ this is at least $\sum_e w_e\langle\alpha,u_e\rangle=\langle\alpha,\sum_e w_eu_e\rangle=0$, by balancing.

We now show that $\ord_x(p|_C)\ge1$ for every $x\in C\cap H$. Since $\cM\subset\Z^n$ and all edge directions are integral, $\ord_x(p|_C)$ is an integer, so it suffices to show that it does not vanish when $|A(x)|\ge2$.

Assume first that $x$ is an interior point of an edge with direction $u$, and that $\ord_x(p|_C)=0$. By the formula above, all monomials of $A(x)$ take the same value of $\langle\alpha,u\rangle$. Every monomial outside $A(x)$ is strictly smaller than the maximum at $x$, so by continuity it remains strictly smaller on a neighbourhood of $x$. Along the edge near $x$, the monomials of $A(x)$ change at the common rate $\langle\alpha,u\rangle$, so they stay equal to each other and maximal. Since $|A(x)|\ge2$, a whole segment of the edge through $x$ lies in $H$, which contradicts the finiteness of $C\cap H$.

Assume next that $x$ is a vertex and that $\ord_x(p|_C)=0$. Choose any $\alpha^\ast\in A(x)$. By the computation above, $0=\ord_x(p|_C)-\langle\alpha^\ast,\sum_ew_eu_e\rangle=\sum_ew_e(\max_{A(x)}\langle\alpha,u_e\rangle-\langle\alpha^\ast,u_e\rangle)$. All the terms of this sum are nonnegative and the weights are positive, so every term vanishes. Hence $\alpha^\ast$ maximizes $\langle\cdot,u_e\rangle$ on $A(x)$ for every edge germ. Since $\alpha^\ast$ was arbitrary, all elements of $A(x)$ take the same value of $\langle\cdot,u_e\rangle$ for each $e$. As in the previous case, the monomials of $A(x)$ stay tied and maximal along every edge germ at $x$. Hence $H$ contains a segment of every edge at $x$, again contradicting finiteness.

It remains to compute $\sum_{x\in C}\ord_x(p|_C)$. On each edge, $p$ is the restriction of a maximum of finitely many affine functions, so it has only finitely many points of non-differentiability. Subdivide $C$ at its vertices and at all these points. This gives finitely many bounded segments and finitely many rays, on each of which $p$ is affine. Only finitely many points of $C$ have nonzero order, namely subdivision points. Hence $\sum_{x\in C}\ord_x(p|_C)$ is the sum, over all subdivision points $x$ and all germs $(e,u)$ at $x$, of $w_e$ times the slope of $p$ along $u$.

Group these terms by the pieces of the subdivision. A bounded segment $[a,b]$ of direction $u$ from $a$ to $b$, on which $p$ has slope $s$ in the direction $u$, contributes $w_es$ at $a$ and $w_e(-s)$ at $b$. These two contributions cancel. A ray piece starting at the last subdivision point of a ray $\rho$ contributes $w_\rho$ times the final slope of $p$ along $\rho$. For large values of the parameter along $\rho$, the maximum in $p$ is attained only by monomials maximizing $\langle\alpha,u_\rho\rangle$, because the terms $t\langle\alpha,u_\rho\rangle$ dominate the bounded remaining parts. Hence the final slope equals $\max_{\alpha\in\cM}\langle\alpha,u_\rho\rangle=h_\Delta(u_\rho)$.

Altogether $\sum_{x\in C}\ord_x(p|_C)=\sum_\rho w_\rho h_\Delta(u_\rho)$. Since every point of the finite set $C\cap H$ has order at least $1$ and every other point has order $0$, we obtain $|C\cap H|\le\sum_{x\in C\cap H}\ord_x(p|_C)=\sum_\rho w_\rho h_\Delta(u_\rho)$.
\end{proof}

\begin{corollary}\label{cor:line}
Let $H$ be a tropical hypersurface with support $\cM$ and Newton polytope $\Delta$, and put $\delta(\Delta)=h_\Delta(1,\dots,1)+\sum_{i=1}^nh_\Delta(-e_i)$. Then $\Tdeg(H)\le\delta(\Delta)$. Let $D_\infty=\max_{\alpha,\beta\in\cM}\|\alpha-\beta\|_\infty$ denote the $\ell^\infty$-diameter of $\cM$. Then $\delta(\Delta)\le\sum_{i=1}^n\width_{e_i}(\cM)\le n\,D_\infty$.
\end{corollary}

\begin{proof}
A tropical line has $n+1$ rays of weight one, with directions $-e_1,\dots,-e_n$ and $(1,\dots,1)$. Applying Theorem~\ref{thm:curve} to every tropical line $L$ with $L\cap H$ finite gives $|L\cap H|\le\delta(\Delta)$, and taking the supremum gives $\Tdeg(H)\le\delta(\Delta)$.

For the second inequality, put $m_i=\min_{\beta\in\cM}\beta_i$ and $M_i=\max_{\beta\in\cM}\beta_i$. Choose $\alpha\in\cM$ with $\sum_i\alpha_i=h_\Delta(1,\dots,1)$. Since $h_\Delta(-e_i)=-m_i$, we have $\delta(\Delta)=\sum_i(\alpha_i-m_i)$. Each term is at most $M_i-m_i=\width_{e_i}(\cM)$. Finally, $M_i-m_i\le D_\infty$ for each $i$, which proves the last inequality.
\end{proof}

The bound $\delta(\Delta)$ is the tropical Bézout number of $H$ and a tropical line, as the next proposition shows.

\begin{proposition}\label{prop:MV}
For every lattice polytope $\Delta\subset\R^n$, we have $n!\,\MV(\Delta,\Sigma,\dots,\Sigma)=\delta(\Delta)$, where $\Sigma$ appears $n-1$ times.
\end{proposition}

\begin{proof}
For $n=1$, both sides equal the length of the interval $\Delta$, so assume $n\ge2$. We use the classical formula expressing a mixed volume with $n-1$ equal entries through the facets of the repeated polytope \cite[Section~5.1]{Schneider2014}. For a polytope $P$ with facets $F$ and outer unit normals $\nu_F$, it reads $n\,\MV(K,P,\dots,P)=\sum_Fh_K(\nu_F)\Vol_{n-1}(F)$, where $\Vol_{n-1}$ is the Euclidean $(n-1)$-dimensional volume.

When $P$ is a lattice polytope, let $u_F\in\Z^n$ be the primitive outer normal of $F$, so that $\nu_F=u_F/\|u_F\|$ and $h_K(\nu_F)=h_K(u_F)/\|u_F\|$. The lattice $u_F^\perp\cap\Z^n$ has covolume $\|u_F\|$ in the hyperplane $u_F^\perp$. Hence $\Vol_{n-1}(F)=\|u_F\|\,\operatorname{vol}_F(F)$, where $\operatorname{vol}_F$ is the volume normalized so that a fundamental domain of the lattice parallel to $F$ has volume $1$. The formula becomes $n\,\MV(K,P,\dots,P)=\sum_Fh_K(u_F)\operatorname{vol}_F(F)$.

For $P=\Sigma$, there are $n+1$ facets. For each $i$, the facet $\conv(0,e_j:j\ne i)$ has outer normal $-e_i$, and the facet $\conv(e_1,\dots,e_n)$ has outer normal $(1,\dots,1)$. Each of these facets is a unimodular $(n-1)$-simplex, so its normalized volume is $1/(n-1)!$. Therefore $n\,\MV(\Delta,\Sigma,\dots,\Sigma)=\delta(\Delta)/(n-1)!$, and multiplying by $(n-1)!$ gives the claim.
\end{proof}

\begin{remark}\label{rem:MV}
The orders $\ord_x(p|_L)$ are the coefficients of the Cartier divisor of $p|_L$ in the sense of Allermann and Rau \cite{AllermannRau2010}, see also \cite{Mikhalkin2006}, and this intersection product coincides with the stable intersection $L\cdot H$ by \cite[Proposition~2.15]{JensenYu2016}. By Theorem~\ref{thm:curve} their sum $\delta(\Delta)$ depends only on the rays of $L$. Combining this with Proposition~\ref{prop:MV}, Corollary~\ref{cor:line} becomes a tropical Bézout bound, consistent with the tropical Bernstein theorem (Theorem~\ref{thm:bernstein}) applied to $\Delta$ and $n-1$ copies of $\Sigma$. The point of Theorem~\ref{thm:curve} is that it bounds set-theoretic intersection points directly, without any transversality hypothesis on $L$ and $H$.
\end{remark}

\begin{example}\label{ex:diam}
Let $n=3$ and $\cM=\{e_1,e_2,e_3,(1,1,1)\}$, with coefficients $c_{e_1}=0.2$, $c_{e_2}=0$, $c_{e_3}=-0.5$ and $c_{(1,1,1)}=0$. Thus $p(x)=\max(x_1+0.2,\,x_2,\,x_3-0.5,\,x_1+x_2+x_3)$.

Let $L$ be the tropical line with a single vertex at $q=(1,1,1)$ and rays $q-te_1$, $q-te_2$, $q-te_3$ and $q+t(1,1,1)$, for $t\ge0$. At $q$ the four monomials take the values $1.2$, $1$, $0.5$ and $3$, so the monomial $(1,1,1)$ dominates and $q\notin H$.

Along the first ray the values are $1.2-t$, $1$, $0.5$ and $3-t$. The monomial $(1,1,1)$ is overtaken by $e_2$ at $t=2$, where the values are $-0.8$, $1$, $0.5$ and $1$. Beyond that point $e_2$ dominates forever, since its value is constant and the other values are either constant and smaller or decreasing.

Along the second ray the values are $1.2$, $1-t$, $0.5$ and $3-t$, and the monomial $(1,1,1)$ is overtaken by $e_1$ at $t=1.8$, where exactly the values of $e_1$ and $(1,1,1)$ coincide. Along the third ray the values are $1.2$, $1$, $0.5-t$ and $3-t$, and again the monomial $(1,1,1)$ is overtaken by $e_1$ at $t=1.8$, with no other tie. Along the fourth ray the value of $(1,1,1)$ is $3+3t$, which stays larger than the others.

Hence $L\cap H$ consists of exactly three points, each in the interior of a ray and each a tie of exactly two monomials. Therefore $|L\cap H|=3$. On the other hand, $\delta(\Delta)=h_\Delta(1,1,1)+h_\Delta(-e_1)+h_\Delta(-e_2)+h_\Delta(-e_3)=3+0+0+0=3$, while $\max_{\alpha,\beta\in\cM}\|\alpha-\beta\|_1=2$. So the $\ell^1$-diameter of $\cM$ does not bound $\Tdeg(H)$, and Corollary~\ref{cor:line} is sharp for this support.
\end{example}

\begin{example}\label{ex:Mminus1}
Let $\cM=\{0,(1,\dots,1)\}\subset\Z^n$ and $p(x)=\max(c,\,x_1+\dots+x_n)$. Then $H$ is the classical hyperplane $x_1+\dots+x_n=c$, with weight one. We take $n\ge2$ when using this family as a counterexample.

Let $L$ be the tropical line with a single vertex at a point $q$ with $q_1+\dots+q_n>c$. Along the ray $q-te_i$ the linear form $x_1+\dots+x_n$ decreases at unit speed, so it crosses the value $c$ exactly once. Along the ray $q+t(1,\dots,1)$ it increases, so there is no crossing. Hence $|L\cap H|=n$, which equals $\delta(\Delta)=n+0$, while $|\cM|-1=1$. The bound $|\cM|-1$ of \cite[Proposition~1.2]{LSS} is therefore specific to classical lines.
\end{example}

For classical lines the upper-envelope argument gives the known estimate and its elementary refinements. We recover \cite[Proposition~1.2]{LSS} and refine it by widths.

\begin{proposition}[cf.\ {\cite[Proposition~1.2]{LSS}}]\label{prop:classical}
Let $\ell=\{x_0+td:t\in\R\}$, with $d\in\R^n\setminus\{0\}$, be a classical line such that $\ell\cap H$ is finite. Then $|\ell\cap H|\le\#\{\langle\alpha,d\rangle:\alpha\in\cM\}-1\le|\cM|-1$. If moreover $d\in\Z^n\setminus\{0\}$, then $|\ell\cap H|\le\width_d(\cM)$. In particular, if $d\in\{0,\pm1\}^n\setminus\{0\}$, for instance when $d$ is the direction of an edge or a ray of a tropical line, then $|\ell\cap H|\le\diam_1(\cM)$, where $\diam_1(\cM)=\max_{\alpha,\beta\in\cM}\|\alpha-\beta\|_1$.
\end{proposition}

\begin{proof}
Consider the function $g(t)=p(x_0+td)=\max_{\alpha\in\cM}(t\langle\alpha,d\rangle+\langle x_0,\alpha\rangle+c_\alpha)$. It is convex and piecewise linear, and on each interval of linearity its slope is one of the numbers $\langle\alpha,d\rangle$. By convexity, the slope increases strictly at each break point. Hence the number of break points is at most the number of distinct values $\langle\alpha,d\rangle$ minus one.

Every $t_0$ with $x_0+t_0d\in H$ is a break point. Indeed, if $g$ were affine near $t_0$, then all monomials attaining the maximum at $t_0$ would have the same slope $\langle\alpha,d\rangle$. As in the proof of Theorem~\ref{thm:curve}, the monomials outside the active set remain strictly smaller near $t_0$, so the active monomials would stay tied on an interval around $t_0$. That interval would lie in $\ell\cap H$, contradicting finiteness. This proves the first two inequalities.

If $d\in\Z^n$, the possible slopes are integers lying between $\min_\alpha\langle\alpha,d\rangle$ and $\max_\alpha\langle\alpha,d\rangle$. A strictly increasing sequence of $m$ such integers satisfies $m-1\le\width_d(\cM)$, which gives the second statement. Primitivity of $d$ is not needed for this argument. Finally, if $d\in\{0,\pm1\}^n\setminus\{0\}$, then $\langle\alpha-\beta,d\rangle\le\sum_i|\alpha_i-\beta_i|=\|\alpha-\beta\|_1$ for all $\alpha,\beta$, so $\width_d(\cM)\le\diam_1(\cM)$. By Definition~\ref{def:curve}, the edge and ray directions of tropical lines lie in $\{0,\pm1\}^n$.
\end{proof}

\section{Stable intersections and lattice indices}\label{sec:stable}

In this paper, a tropical cycle in $N_\R$ is a pure, rational, weighted and balanced polyhedral complex with positive integer facet weights. The empty cycle is also allowed. The support criterion below is used only for these effective cycles; extending the theory to signed cycles requires retaining cancellation in the weights. For a cell $\sigma$ we write $N_\sigma=\spn(\sigma-\sigma)\cap N$ for its tangent lattice. By Lemma~\ref{lem:Nsigma} this is a saturated lattice of rank $\dim\sigma$. For a point $\omega$ we write $\link_\omega(X)$ for the local fan of $X$ at $\omega$, that is, the set of vectors $u$ such that $\omega+\varepsilon u\in X$ for all sufficiently small $\varepsilon>0$, with the induced weights.

Following \cite[Definition~2.4]{JensenYu2016}, the stable intersection $X\cdot Y$ of two tropical cycles $X$ and $Y$ is supported on the set of points $\omega$ such that the support of the Minkowski difference $\link_\omega(X)-\link_\omega(Y)$ is the whole of $N_\R$. A cell $\gamma$ of $X\cdot Y$ receives the multiplicity $\mult_{X\cdot Y}(\gamma)=\sum\mult_X(\sigma)\,\mult_Y(\tau)\,[N:N_\sigma+N_\tau]$. Here the sum runs over the pairs of facets $\sigma$ of $\link_\gamma(X)$ and $\tau$ of $\link_\gamma(Y)$ such that $v\in\sigma-\tau$, for a fixed generic vector $v$, and the result does not depend on $v$. This is the fan displacement rule of Fulton and Sturmfels \cite{FultonSturmfels1997}, used also in \cite{RGST2005,Mikhalkin2006}.

By \cite[Theorem~2.13]{JensenYu2016}, $X\cdot Y$ is again a tropical cycle, either empty or of codimension $\operatorname{codim}X+\operatorname{codim}Y$. Commutativity follows from the displacement definition, and associativity is \cite[Theorem~2.14]{JensenYu2016}. The product agrees with the tropical Cartier-divisor intersection product of Allermann and Rau \cite{AllermannRau2010}; see \cite[Proposition~2.15]{JensenYu2016}. When $X$ and $Y$ have positive weights, every term in the formula for $\mult_{X\cdot Y}(\gamma)$ is positive, so $X\cdot Y$ has positive weights. In particular, a nonempty stable intersection of $k$ tropical hypersurfaces always has codimension exactly $k$. The stable intersection is local: its restriction to a neighbourhood of a point depends only on the links of $X$ and $Y$ at that point.

\begin{lemma}\label{lem:transverse}
Let $\sigma$ and $\tau$ be facets of tropical cycles $X$ and $Y$, let $\omega\in\relint\sigma\cap\relint\tau$, and assume that $\spn(\sigma-\sigma)+\spn(\tau-\tau)=N_\R$. Then, near $\omega$, the stable intersection $X\cdot Y$ is supported on $\sigma\cap\tau$, and its weight there is $\mult_X(\sigma)\,\mult_Y(\tau)\,[N:N_\sigma+N_\tau]$.
\end{lemma}

\begin{proof}
If a cell $\kappa$ of $X$ contains $\omega$, then $\sigma\cap\kappa$ is a common face of $\sigma$ and $\kappa$. This face contains a relative interior point of $\sigma$, so it equals $\sigma$, and $\sigma$ is a face of $\kappa$. Since $\sigma$ is maximal, $\kappa=\sigma$. Hence $X$ coincides with $\sigma$ near $\omega$, and $\link_\omega(X)$ is the linear space $\spn(\sigma-\sigma)$ with weight $\mult_X(\sigma)$, and likewise $\link_\omega(Y)=\spn(\tau-\tau)$ with weight $\mult_Y(\tau)$. The same holds at all points of $\sigma\cap\tau$ near $\omega$. The Minkowski difference of the two linear spaces is their sum, which is $N_\R$ by assumption, so these points belong to $\supp(X\cdot Y)$. Points near $\omega$ outside $\sigma\cap\tau$ do not lie in $X\cap Y$, and the stable intersection is contained in $X\cap Y$. In the multiplicity formula there is exactly one pair of facets, namely the two linear spaces, and every $v$ lies in their difference $N_\R$. This gives the stated weight.
\end{proof}

When $\dim\sigma+\dim\tau=n$, Lemma~\ref{lem:index} shows that $[N:N_\sigma+N_\tau]=|\det(u_1,\dots,u_a,v_1,\dots,v_b)|$, where $u_1,\dots,u_a$ and $v_1,\dots,v_b$ are $\Z$-bases of $N_\sigma$ and $N_\tau$. For several hypersurfaces, the transverse multiplicity is described by the following normal lattice. The formula of Proposition~\ref{prop:mult} is the transversal case of the definition of intersection weights given by Bertrand and Bihan \cite[Definition~5.2]{BertrandBihan2013}; see also \cite[Lemma~5.4]{BertrandBihan2013} for its interpretation as a mixed volume. What we add is a proof that this weight agrees with the iterated stable intersection product of Jensen and Yu. For a subgroup $\Lambda\subset M$, its saturation is $\sat\Lambda=(\Lambda\otimes\R)\cap M$.

\begin{proposition}\label{prop:mult}
Let $\sigma_1,\dots,\sigma_r$ be facets of the tropical hypersurfaces $V(p_1),\dots,V(p_r)$, with weights $w_1,\dots,w_r$ and primitive normal vectors $n_1,\dots,n_r\in M$. Assume that $n_1,\dots,n_r$ are linearly independent, and let $x\in\relint\sigma_1\cap\dots\cap\relint\sigma_r$. Put $\Lambda=\Z n_1+\dots+\Z n_r\subset M$. Then, near $x$, the stable intersection $V(p_1)\cdots V(p_r)$ is supported on the affine piece $x+(n_1^\perp\cap\dots\cap n_r^\perp)$, and its weight there is $m(x)=w_1\cdots w_r\,[\sat\Lambda:\Lambda]$. This index equals the greatest common divisor of the $r\times r$ minors of the $n\times r$ matrix with columns $n_1,\dots,n_r$. In particular, when $r=n$ we have $m(x)=w_1\cdots w_n\,|\det(n_1,\dots,n_n)|=|\det(w_1n_1,\dots,w_nn_n)|$.
\end{proposition}

\begin{proof}
For $j=1,\dots,r$ put $\Lambda_j=\Z n_1+\dots+\Z n_j$ and $A_j=N\cap n_1^\perp\cap\dots\cap n_j^\perp$. The lattice $A_j$ is saturated, because it is the intersection of $N$ with a linear subspace. Its annihilator in $M$ is $A_j^\perp=\sat\Lambda_j$. Indeed, $A_j^\perp$ is saturated and its span is $\spn(n_1,\dots,n_j)$, so it is the full intersection of that span with $M$.

We show by induction on $j$ that, near $x$, the stable intersection $X_j=V(p_1)\cdots V(p_j)$ is supported on $x+\spn(A_j)$ with weight $m_j=w_1\cdots w_j[\sat\Lambda_j:\Lambda_j]$.

For $j=1$, near $x$ the hypersurface $V(p_1)$ coincides with $\sigma_1$, of weight $w_1$, and $\sat\Lambda_1=\Lambda_1=\Z n_1$ because $n_1$ is primitive.

Assume the claim for $j<r$. Near $x$, $X_j$ is a single maximal cell with tangent lattice $A_j$, and $V(p_{j+1})$ is the facet $\sigma_{j+1}$, with tangent lattice $B=N\cap n_{j+1}^\perp$. Since $n_{j+1}$ does not lie in $\spn(n_1,\dots,n_j)$, the hyperplane $n_{j+1}^\perp$ does not contain $\spn(A_j)$. Hence $\spn(A_j)+n_{j+1}^\perp=N_\R$. Lemma~\ref{lem:transverse} shows that near $x$ the stable intersection $X_{j+1}=X_j\cdot V(p_{j+1})$ is supported on $x+\spn(A_j\cap B)=x+\spn(A_{j+1})$, with weight $m_{j+1}=m_j\,w_{j+1}\,[N:A_j+B]$.

By Lemma~\ref{lem:dual}, $[N:A_j+B]=[(A_j\cap B)^\perp:A_j^\perp+B^\perp]$. Here $A_j\cap B=A_{j+1}$, so $(A_j\cap B)^\perp=\sat\Lambda_{j+1}$. Also $A_j^\perp=\sat\Lambda_j$, and $B^\perp=\Z n_{j+1}$, because $n_{j+1}$ is primitive. Hence $[N:A_j+B]=[\sat\Lambda_{j+1}:\sat\Lambda_j+\Z n_{j+1}]$.

Since $n_{j+1}\notin\spn(\sat\Lambda_j)$, the sums $\sat\Lambda_j+\Z n_{j+1}$ and $\Lambda_j+\Z n_{j+1}=\Lambda_{j+1}$ are direct. Therefore the quotient $(\sat\Lambda_j\oplus\Z n_{j+1})/(\Lambda_j\oplus\Z n_{j+1})$ is isomorphic to $\sat\Lambda_j/\Lambda_j$. The multiplicativity of indices along the chain $\Lambda_{j+1}\subset\sat\Lambda_j+\Z n_{j+1}\subset\sat\Lambda_{j+1}$ then gives $[\sat\Lambda_{j+1}:\Lambda_{j+1}]=[\sat\Lambda_{j+1}:\sat\Lambda_j+\Z n_{j+1}]\,[\sat\Lambda_j:\Lambda_j]$. Substituting into $m_{j+1}=m_jw_{j+1}[N:A_j+B]$ yields $m_{j+1}=w_1\cdots w_{j+1}[\sat\Lambda_{j+1}:\Lambda_{j+1}]$, which completes the induction.

For the description by minors, let $A$ be the integer $n\times r$ matrix with columns $n_1,\dots,n_r$, so that $\Lambda=A(\Z^r)$. By the Smith normal form there are $P\in GL_n(\Z)$ and $Q\in GL_r(\Z)$ such that $PAQ$ is the diagonal matrix with nonzero diagonal entries $d_1,\dots,d_r$ placed in the first $r$ rows. Applying $P$, which preserves both the lattice and saturation, we may assume $\Lambda=d_1\Z\oplus\dots\oplus d_r\Z\oplus0$. Then $\sat\Lambda=\Z^r\oplus0$, and the index is $|d_1\cdots d_r|$. This product is the greatest common divisor of the $r\times r$ minors of $A$. Indeed, the ideal generated by these minors is invariant under multiplication by unimodular matrices on either side: the Cauchy--Binet formula gives one containment of ideals, and applying the inverse unimodular transformations gives the other. When $r=n$ this greatest common divisor is $|\det A|$.
\end{proof}

\begin{example}\label{ex:notsum}
The index $[N:N_{\sigma_1}+\dots+N_{\sigma_r}]$ of the sum of the tangent lattices does not compute the multiplicity in general. Let $p_1=\max(0,x)$, $p_2=\max(0,y)$ and $p_3=\max(0,x+y+2z)$ in $\R^3$. The three hypersurfaces are the planes $x=0$, $y=0$ and $x+y+2z=0$, with primitive normals $(1,0,0)$, $(0,1,0)$, $(1,1,2)$ and weights one.

The tangent lattice of the first plane contains $e_2$ and $e_3$, and that of the second contains $e_1$. So the sum of the three tangent lattices is $\Z^3$, and its index is $1$. On the other hand, Proposition~\ref{prop:mult} gives multiplicity $|\det((1,0,0),(0,1,0),(1,1,2))|=2$ at the origin. This agrees with Theorem~\ref{thm:bernstein}: the three Newton polytopes are the segments $[0,e_1]$, $[0,e_2]$ and $[0,(1,1,2)]$, and $3!$ times their mixed volume is the volume of the parallelepiped they span, namely $2$.
\end{example}

\section{Mixed cells and the tropical Bernstein theorem}\label{sec:bernstein}

Let $q_1,\dots,q_n$ be tropical polynomials on $\R^n$ with supports $\cM_1,\dots,\cM_n$ and Newton polytopes $P_1,\dots,P_n$. For $x\in\R^n$, let $A_i(x)\subset\cM_i$ be the set of monomials attaining the maximum in $q_i(x)$. The cells $\conv(A_1(x))+\dots+\conv(A_n(x))$, for $x\in\R^n$, form a regular mixed subdivision of $P_1+\dots+P_n$. This subdivision is induced by the coefficients, and it is dual to the common refinement of the polyhedral complexes defined by the $q_i$ \cite{HuberSturmfels1995,MaclaganSturmfels2015}. A cell $C_1+\dots+C_n$ of it is called fully mixed if each $C_i$ is an edge and $\dim(C_1+\dots+C_n)=n$.

We call $x\in\R^n$ a \emph{transverse point} of $V(q_1)\cap\dots\cap V(q_n)$ if, for each $i$, $x$ lies in the relative interior of a facet $\sigma_i$ of $V(q_i)$, and the primitive normal vectors $n_1,\dots,n_n$ of these facets are linearly independent.

The following standard local formula is recorded with proof; see \cite{BertrandBihan2013,HuberSturmfels1995,MaclaganSturmfels2015} for the intersection and mixed-subdivision framework.

\begin{theorem}\label{thm:local}
Let $x$ be a transverse point of $V(q_1)\cap\dots\cap V(q_n)$, and let $E_i=[u_i,v_i]\subset P_i$ be the edge of the subdivision of $P_i$ dual to $\sigma_i$. Then $x$ is an isolated point of $V(q_1)\cap\dots\cap V(q_n)$ and belongs to the stable intersection $V(q_1)\cdots V(q_n)$. The cell of the mixed subdivision dual to $x$ is the fully mixed parallelotope $E_1+\dots+E_n$, and the multiplicity of $x$ is $m(x)=|\det(v_1-u_1,\dots,v_n-u_n)|=\Vol(E_1+\dots+E_n)$.
\end{theorem}

\begin{proof}
Near $x$, each $V(q_i)$ coincides with the hyperplane piece $x+n_i^\perp$, because $x$ lies in the relative interior of the maximal cell $\sigma_i$. Since $n_1,\dots,n_n$ are linearly independent, these hyperplanes meet only at $x$, so $x$ is isolated in the intersection. Proposition~\ref{prop:mult}, applied with $r=n$, shows that $x$ belongs to the stable intersection with multiplicity $m(x)=|\det(w_1n_1,\dots,w_nn_n)|$, where $w_i$ is the weight of $\sigma_i$.

Since $x\in\relint\sigma_i$, the set $\conv(A_i(x))$ is the edge $E_i$ dual to $\sigma_i$. So the cell of the mixed subdivision dual to $x$ is $E_1+\dots+E_n$. By the conventions in the introduction, $v_i-u_i=w_in_i$. The Minkowski sum $E_1+\dots+E_n$ is therefore the parallelotope $u_1+\dots+u_n+\{\sum_it_iw_in_i:0\le t_i\le1\}$. Its dimension is $n=\sum_i\dim E_i$, since the vectors $w_in_i$ are linearly independent, so the cell is fully mixed. Its volume is $|\det(w_1n_1,\dots,w_nn_n)|=|\det(v_1-u_1,\dots,v_n-u_n)|$, which equals $m(x)$.
\end{proof}

\begin{theorem}[tropical Bernstein theorem; \cite{HuberSturmfels1995}, {\cite[Corollary~4.5 and Theorem~4.6]{JensenYu2016}}, {\cite[\S4.6]{MaclaganSturmfels2015}}]\label{thm:bernstein}
The stable intersection $V(q_1)\cdots V(q_n)$ is finite, and its total multiplicity equals $n!\,\MV(P_1,\dots,P_n)$. Moreover, if every point of the set-theoretic intersection $V(q_1)\cap\dots\cap V(q_n)$ is transverse, then the set-theoretic intersection coincides with the support of the stable intersection, and $\sum_xm(x)$ equals both the sum of the volumes of the fully mixed cells of the mixed subdivision and $n!\,\MV(P_1,\dots,P_n)$.
\end{theorem}

\begin{proof}
The stable intersection of $n$ hypersurfaces is empty or of codimension $n$ by \cite[Theorem~2.13]{JensenYu2016}, hence it is a finite set of points. The equality of its total multiplicity with $n!\,\MV(P_1,\dots,P_n)$, for arbitrary real coefficients, is the tropical Bernstein theorem of the references cited in the statement. In the constant-coefficient case it is \cite[Corollary~4.5]{JensenYu2016}, and the general case is treated in \cite{HuberSturmfels1995} and \cite[\S4.6]{MaclaganSturmfels2015}. We do not reprove this general statement. We give a complete proof of the second assertion, which is the case used in Theorem~\ref{thm:local} and in the examples, and which identifies the local contributions.

Assume that every point of the set-theoretic intersection is transverse. The stable intersection is contained in the set-theoretic intersection, and Theorem~\ref{thm:local} gives the reverse inclusion together with the local multiplicities. Every intersection point is a vertex of the common refinement of the polyhedral complexes defined by the $q_i$, so there are finitely many such points.

We compute their total multiplicity by comparing volume coefficients. If $P_1+\dots+P_n$ is not full-dimensional, then its volume polynomial is identically zero. In that case the edge directions of the $P_i$ span a proper subspace, so independent dual edges cannot occur; the transverse intersection is therefore empty and the assertion holds. Otherwise consider the regular mixed subdivision of $P_1+\dots+P_n$. For all positive $\lambda_1,\dots,\lambda_n$, the cells $\lambda_1C_1+\dots+\lambda_nC_n$ form a subdivision of $\lambda_1P_1+\dots+\lambda_nP_n$. Indeed, multiplying the support and the coefficients of $q_i$ by $\lambda_i$ multiplies $q_i$ by $\lambda_i$, which does not change the active sets $A_i(x)$; see also \cite[Section~2]{HuberSturmfels1995}. Additivity of volume shows that $\Vol(\lambda_1P_1+\dots+\lambda_nP_n)$ equals the sum of $\Vol(\lambda_1C_1+\dots+\lambda_nC_n)$ over the full-dimensional cells $C=C_1+\dots+C_n$. Comparing the coefficients of $\lambda_1\cdots\lambda_n$ in these polynomial identities gives $n!\,\MV(P_1,\dots,P_n)=\sum_Cn!\,\MV(C_1,\dots,C_n)$.

If some $C_i$ is a point, the corresponding mixed volume vanishes. If every $C_i$ has positive dimension, then the vertex $x$ of the common refinement dual to $C$ satisfies $|A_i(x)|\ge2$ for every $i$, so $x$ lies in every $V(q_i)$. By hypothesis $x$ is transverse, so each $C_i$ is the edge $E_i$ dual to the facet of $V(q_i)$ containing $x$, and the edge directions are independent. Consequently $n!\,\MV(E_1,\dots,E_n)=\Vol(E_1+\dots+E_n)=m(x)$ by Theorem~\ref{thm:local}. Conversely, every transverse intersection point is dual to such a cell. Summing over the cells proves that $\sum_xm(x)$ equals the sum of the volumes of the fully mixed cells and equals $n!\,\MV(P_1,\dots,P_n)$.
\end{proof}

\begin{definition}\label{def:L}
Fix $1\le r\le n$. Let $h_j(x)=\max(a_{j0},x_1+a_{j1},\dots,x_n+a_{jn})$ for $j=1,\dots,n-r$, with real coefficients $a_{jl}$. Each $h_j$ has Newton polytope $\Sigma$. The stable intersection $L=V(h_1)\cdots V(h_{n-r})$ is a tropical linear space of dimension $r$ obtained as a stable intersection of standard tropical hyperplanes; when the coefficients are generic, we call this a generic linear section. Throughout the rest of the paper, linear sections are restricted to this presentation. We make no assertion that it presents every tropical linear space in the general valuated-matroid sense \cite{Speyer2008}. For $r=n$ the empty stable product is $L=N_\R$ with weight one. A point $x\in L$ is a \emph{regular point} of $L$ if it lies in the relative interior of a facet $\tau_j$ of $V(h_j)$ for every $j$, and the primitive normals $n'_1,\dots,n'_{n-r}$ of these facets are linearly independent. When $r=n$, every point is regular and the empty lattice index below is one. Otherwise, by Proposition~\ref{prop:mult}, near a regular point $x$ the space $L$ is the affine piece $x+(n'^{\perp}_1\cap\dots\cap n'^{\perp}_{n-r})$, of dimension $r$, with weight $m_L(x)=[\sat\Lambda':\Lambda']$, where $\Lambda'=\sum_j\Z n'_j$.
\end{definition}

\begin{corollary}\label{cor:CI}
Let $X=V(p_1)\cdots V(p_r)$ and let $L$ be a linear section of dimension $r$ as in Definition~\ref{def:L}. Then $X\cdot L$ is finite, and its total multiplicity is $n!\,\MV(\Delta_1,\dots,\Delta_r,\Sigma,\dots,\Sigma)$, where $\Sigma$ appears $n-r$ times. At a transverse point of the $n$ hypersurfaces $V(p_1),\dots,V(p_r),V(h_1),\dots,V(h_{n-r})$, the multiplicity equals the volume of the dual fully mixed cell of the mixed subdivision of $\Delta_1+\dots+\Delta_r+\Sigma+\dots+\Sigma$.
\end{corollary}

\begin{proof}
By associativity of the stable intersection, $X\cdot L=V(p_1)\cdots V(p_r)\cdot V(h_1)\cdots V(h_{n-r})$, a stable intersection of $n$ hypersurfaces with Newton polytopes $\Delta_1,\dots,\Delta_r,\Sigma,\dots,\Sigma$. Theorems~\ref{thm:bernstein} and~\ref{thm:local} now give both statements.
\end{proof}

\section{Tropical Systems with Excess Equations}\label{sec:redundant}

In this section $X=V(p_1)\cap\dots\cap V(p_k)$ is the set-theoretic intersection of $k$ tropical hypersurfaces, which is a tropical prevariety. As recalled in Section~\ref{sec:stable}, a nonempty stable intersection of $k$ hypersurfaces always has codimension $k$. We fix $1\le r\le\min(n,k)$ and count the points where the local normal rank is $r$. The redundant case is $r<k$. We do not assume that the whole prevariety is pure-dimensional. In particular, the estimate below is for its specified regular transverse locus, not for all singular or isolated points.

\begin{definition}
A point $x\in X$ is \emph{regular} if, for every $j\in\{1,\dots,k\}$, $x$ lies in the relative interior of a facet $\sigma_j$ of $V(p_j)$. We then write $w_j$ and $n_j$ for the weight and the primitive normal vector of $\sigma_j$, and $s(x)$ for the rank of the family $n_1,\dots,n_k$.
\end{definition}

\begin{proposition}\label{prop:reduction}
Let $x\in X$ be a regular point and put $s=s(x)$.

(a) Near $x$, $X$ is an open piece of the affine subspace $x+(n_1^\perp\cap\dots\cap n_k^\perp)$, which has codimension $s$.

(b) Let $I=\{i_1,\dots,i_s\}$ be such that $n_{i_1},\dots,n_{i_s}$ form a basis of $\spn(n_1,\dots,n_k)$. Then, near $x$, $X=V(p_{i_1})\cap\dots\cap V(p_{i_s})$. Moreover, the stable intersection $V(p_{i_1})\cdots V(p_{i_s})$ is supported on $X$ near $x$, with weight $m_I(x)=w_{i_1}\cdots w_{i_s}[\sat\Lambda_I:\Lambda_I]$, where $\Lambda_I=\sum_{i\in I}\Z n_i$.

(c) The weight $m_I(x)$ depends on the choice of $I$ in general.
\end{proposition}

\begin{proof}
For (a): since $\sigma_j$ is a maximal cell of $V(p_j)$ and $x$ lies in its relative interior, there is an open neighbourhood $U_j$ of $x$ with $U_j\cap V(p_j)=U_j\cap(x+n_j^\perp)$. On $U=U_1\cap\dots\cap U_k$ we get $U\cap X=U\cap(x+n_1^\perp\cap\dots\cap n_k^\perp)$. The codimension of $n_1^\perp\cap\dots\cap n_k^\perp$ is the dimension of $\spn(n_1,\dots,n_k)$, namely $s$.

For (b): since $n_{i_1},\dots,n_{i_s}$ span the same space as $n_1,\dots,n_k$, we have $n_{i_1}^\perp\cap\dots\cap n_{i_s}^\perp=n_1^\perp\cap\dots\cap n_k^\perp$. Hence, on $U$, $V(p_{i_1})\cap\dots\cap V(p_{i_s})$ coincides with $X$ by the argument of (a). The statement about the stable intersection and its weight is Proposition~\ref{prop:mult}, applied to the facets $\sigma_{i_1},\dots,\sigma_{i_s}$.

For (c): let $n=1$, $k=2$, $p_1(x)=\max(0,x)$ and $p_2(x)=\max(0,2x)$. Both hypersurfaces are the point $0$, so $X=\{0\}$, every point of $X$ is regular, and $s=1$. The facet of $V(p_1)$ has weight $1$, and that of $V(p_2)$ has weight $2$, since the dual edge $[0,2]$ has lattice length $2$. Hence $m_{\{1\}}(0)=1$ and $m_{\{2\}}(0)=2$. The same phenomenon occurs in every dimension, for instance with $\max(0,x_1)$ and $\max(0,2x_1)$ in $\R^n$.
\end{proof}

Proposition~\ref{prop:reduction}(c) shows that choosing a locally independent subsystem does not produce a subsystem-independent weight. It does not exclude other weight constructions based on additional data. The next example shows that the total stable degree of a fixed subsystem need not equal its weighted contribution on the redundant intersection.

\begin{example}\label{ex:redundant}
In $\R^2$ with coordinates $(x,y)$, let $p_1=\max(0,x)\odot\max(0,x-5)$ and $p_2=\max(0,x)\odot\max(0,x+5)$, where $\odot$ denotes tropical multiplication, that is, ordinary addition. Then $p_1=\max(0,x-5,x,2x-5)$, which reduces to $\max(0,x,2x-5)$. Its hypersurface is $V(p_1)=\{x=0\}\cup\{x=5\}$, each line with weight one, and its Newton polytope is $\Delta_1=[0,2e_1]$. Similarly, $V(p_2)=\{x=0\}\cup\{x=-5\}$. Hence $X=V(p_1)\cap V(p_2)=\{x=0\}$. Here $r=1<k=2$, every point of $X$ is regular, and $X=V(p_1)$ near each of its points.

Let $L$ be a generic tropical line. It meets the vertical line $\{x=0\}$ in a single point, at which $V(p_1)\cdot L$ has weight one. However, $2!\,\MV(\Delta_1,\Sigma)=\delta(\Delta_1)=h_{\Delta_1}(1,1)+h_{\Delta_1}(-1,0)+h_{\Delta_1}(0,-1)=2+0+0=2$. The second point of $V(p_1)\cdot L$ lies on the component $\{x=5\}$, which is not contained in $X$. Here the local subsystem weight is $m_{\{1\}}=1$, and $L$ is a tropical line, so $m_L=1$. The line $L$ meets $\{x=0\}$ on a ray of direction $-e_1$ or $(1,1)$, whose primitive normal is $e_2$ or $(1,-1)$. Together with the normal $e_1$ of $\{x=0\}$ this gives $[M:\sat\Lambda_{\{1\}}\oplus\sat\Lambda']=|\det(e_1,e_2)|=1$ or $|\det(e_1,(1,-1))|=1$, so the multiplicity of the point in $V(p_1)\cdot L$ is $1$. Thus $\sum_{x\in X\cap L}m_{\{1\}}(x)=1$, whereas the full subsystem degree is $2$. The same failure holds for the other one-equation subsystem. This disproves a universal equality between the count on $X$ with these local subsystem weights and the mixed-volume degree of a fixed subsystem. It does not rule out formulas using further information about the equations.
\end{example}

Let $L=V(h_1)\cdots V(h_{n-r})$ be a linear section of dimension $r$, as in Definition~\ref{def:L}. We say that $x\in X\cap L$ is a \emph{transverse point of rank $r$} if three conditions hold: $x$ is a regular point of $X$ with $s(x)=r$; $x$ is a regular point of $L$; and the $n$ vectors consisting of a basis $\{n_i\}_{i\in I}$ of $\spn(n_1,\dots,n_k)$ together with the normals $n'_1,\dots,n'_{n-r}$ of Definition~\ref{def:L} are linearly independent. This last condition does not depend on the choice of the basis, and it says that the tangent spaces of $X$ and $L$ at $x$ are complementary. We denote by $T(X,L)$ the set of transverse points of rank $r$. For a subset $I\subset\{1,\dots,k\}$ of cardinality $r$, we denote by $T_I$ the set of $x\in T(X,L)$ at which $\{n_i\}_{i\in I}$ is a basis of $\spn(n_1,\dots,n_k)$. Then $T(X,L)$ is the union of the sets $T_I$. We write $\mathcal I_r$ for the family of those $I$ with $T_I$ nonempty, and $\Delta_I$ for the list of polytopes $\Delta_i$, $i\in I$.

This estimate is an application of tropical Bernstein to each independent subsystem, followed by a finite union bound.

\begin{theorem}\label{thm:bezout}
With the notation above, $T(X,L)$ is finite. For every $I\in\mathcal I_r$ we have $$\sum_{x\in T_I}m_I(x)m_L(x)\le n!\,\MV(\Delta_I,\Sigma,\dots,\Sigma),$$ where $\Sigma$ appears $n-r$ times. Consequently, $|T(X,L)|\le\sum_{I\in\mathcal I_r}n!\,\MV(\Delta_I,\Sigma,\dots,\Sigma)$, and this sum is at most $\binom kr\max_{|I|=r}n!\,\MV(\Delta_I,\Sigma,\dots,\Sigma)$. If a single $I$ satisfies $T_I=T(X,L)$, then $|T(X,L)|\le n!\,\MV(\Delta_I,\Sigma,\dots,\Sigma)$.
\end{theorem}

\begin{proof}
Fix $I\in\mathcal I_r$ and $x\in T_I$. For $i\in I$, the point $x$ lies in the relative interior of the facet $\sigma_i$ of $V(p_i)$. For each $j$, it lies in the relative interior of the facet $\tau_j$ of $V(h_j)$. By the definition of $T(X,L)$, the $n$ primitive normals $n_i$ ($i\in I$) and $n'_j$ are linearly independent. Hence $x$ is a transverse point, in the sense of Section~\ref{sec:bernstein}, of the $n$ hypersurfaces $V(p_i)$ ($i\in I$) and $V(h_1),\dots,V(h_{n-r})$, whose Newton polytopes are $\Delta_I$ and $n-r$ copies of $\Sigma$.

By Theorem~\ref{thm:local}, $x$ belongs to the stable intersection $Z_I$ of these $n$ hypersurfaces, with multiplicity $\prod_{i\in I}w_i\,|\det(n_i,n'_j)|$, since the facets of the standard tropical hyperplanes have weight one. To compare this with $m_I(x)m_L(x)$, note that $\Lambda_I\oplus\Lambda'\subset\sat\Lambda_I\oplus\sat\Lambda'\subset M$, and the sums are direct because the $n$ vectors are independent. By multiplicativity of indices, $|\det(n_i,n'_j)|=[M:\Lambda_I\oplus\Lambda']$, and this equals $[M:\sat\Lambda_I\oplus\sat\Lambda']\,[\sat\Lambda_I:\Lambda_I]\,[\sat\Lambda':\Lambda']$. It is therefore at least $[\sat\Lambda_I:\Lambda_I]\,[\sat\Lambda':\Lambda']$. In fact, the exact multiplicity is $m_I(x)m_L(x)[M:\sat\Lambda_I\oplus\sat\Lambda']$, which is at least $m_I(x)m_L(x)$.

By Theorem~\ref{thm:bernstein}, $Z_I$ is finite, all its nonzero multiplicities are positive, and its total multiplicity is $n!\,\MV(\Delta_I,\Sigma,\dots,\Sigma)$. Distinct points of $T_I$ are distinct points of $Z_I$, so summing over $x\in T_I$ gives the first inequality. Since $m_I(x)m_L(x)\ge1$, this yields $|T_I|\le n!\,\MV(\Delta_I,\Sigma,\dots,\Sigma)$. There are finitely many $I$, so the inclusion of each $T_I$ in the finite support of $Z_I$ proves finiteness of $T(X,L)$. Summing over $I\in\mathcal I_r$ and using $T(X,L)=\bigcup_IT_I$ gives the second inequality. If the transverse locus is empty, the sum is empty and the claimed bounds are immediate. The third holds because $\mathcal I_r$ has at most $\binom kr$ elements, and the last statement is the case where the union consists of a single $T_I$.
\end{proof}

\begin{remark}
When $r=1$, every individual normal is a basis of the rank-one span, so $T_{\{i\}}=T(X,L)$ for every $i$. The last assertion of the theorem therefore gives the stronger bound $|T(X,L)|\le\min_i n!\,\MV(\Delta_i,\Sigma,\dots,\Sigma)$. The sum over $\mathcal I_r$ is what the argument above provides; we do not know whether the maximum alone suffices in general (Question~\ref{q:max}). Example~\ref{ex:redundant} shows that the inequality of Theorem~\ref{thm:bezout} can be strict.
\end{remark}

\section{Questions}\label{sec:questions}

\begin{question}\label{q:max}
In the setting of Theorem~\ref{thm:bezout}, does the inequality $|T(X,L)|\le\max_{|I|=r}n!\, \\ \MV(\Delta_I,\Sigma,\dots,\Sigma)$ always hold? Ren's theorem \cite[Theorem~3.3]{Ren2024} states that every connected component of an intersection of tropical hypersurfaces contains a point of their stable intersection unless that stable intersection is empty. This theorem alone does not imply the proposed bound: a connected component of a subsystem intersection can contain several points of $T(X,L)$.
\end{question}

\begin{question}
For which polytopes $\Delta$ is the bound $\Tdeg(H)\le\delta(\Delta)$ attained for generic coefficients? For a fixed line $L$ with finite intersection, Theorem~\ref{thm:curve} shows that $|L\cap H|=\delta(\Delta)$ if and only if $\ord_x(p|_L)=1$ for every $x\in L\cap H$. At a transverse intersection of an edge of $L$ with a facet of $H$ of weight larger than one, the order is larger than one, so that particular line cannot realize equality. The existence of a heavier facet somewhere in $H$ does not, by itself, exclude the possibility of equality for another line that avoids that facet.
\end{question}

\begin{question}
The contour of the amoeba of a hypersurface is the set of critical values of the logarithmic map restricted to the hypersurface \cite{LSS}, and the spine of an amoeba is a tropical hypersurface \cite{PassareRullgard2004}. Do the bounds of Section~\ref{sec:degree}, applied to tropical lines and to the spine, lead to bounds for intersections of amoeba contours with tropical-type curves, in analogy with the boundary--spine comparison in \cite[Proposition~1.4]{LSS} and the contour bounds in \cite[Propositions~1.10--1.11]{LSS}? No transfer from a spine count to a contour count is asserted here.
\end{question}

\appendix
\section{Lattice-theoretic facts}\label{app:lattice}

The saturation and determinant facts below are standard lattice theory; for general background see \cite{Cassels1971}. The dual index identity is explicitly attributed before its statement. Proofs are included for completeness.

\begin{lemma}\label{lem:Nsigma}
Let $\sigma\subset N_\R$ be a nonempty rational polyhedron. Then $N_\sigma=\spn(\sigma-\sigma)\cap N$ is a saturated subgroup of $N$, and it is a lattice of rank $\dim\sigma$ in $\spn(\sigma-\sigma)$.
\end{lemma}

\begin{proof}
A rational polyhedron is defined by finitely many inequalities $\langle m_i,x\rangle\ge a_i$ with $m_i\in M$. Its affine hull is cut out by those inequalities that hold with equality on the whole polyhedron. Hence $\spn(\sigma-\sigma)$ is the common kernel of finitely many vectors of $M$, and it has a basis $v_1,\dots,v_d$ of rational vectors, with $d=\dim\sigma$. Multiplying each $v_i$ by a common denominator gives vectors $u_1,\dots,u_d\in N\cap\spn(\sigma-\sigma)=N_\sigma$ which still span $\spn(\sigma-\sigma)$.

The set $N_\sigma$ is a subgroup of $N$, being an intersection of two subgroups of $N_\R$, and it is discrete because $N$ is discrete. A discrete subgroup of a real vector space is free of rank at most the dimension, and $N_\sigma$ contains $d$ linearly independent vectors. So it is a lattice of rank $d$ in $\spn(\sigma-\sigma)$. Finally, if $kx\in N_\sigma$ with $k$ a nonzero integer and $x\in N$, then $x\in\spn(\sigma-\sigma)$, so $x\in N_\sigma$; hence $N_\sigma$ is saturated.
\end{proof}

\begin{lemma}\label{lem:index}
Let $A,B\subset N$ be subgroups with $\spn A+\spn B=N_\R$. Then $A+B$ has finite index in $N$. If in addition $\spn A\cap\spn B=0$, and $u_1,\dots,u_a$ and $v_1,\dots,v_b$ are $\Z$-bases of $A$ and $B$, then $a+b=n$ and $[N:A+B]=|\det(u_1,\dots,u_a,v_1,\dots,v_b)|$.
\end{lemma}

\begin{proof}
The subgroup $A+B$ spans $N_\R$, so it contains $n$ linearly independent vectors and has rank $n$. A subgroup of rank $n$ of $N\cong\Z^n$ has finite index, because the quotient is finitely generated of rank zero, hence finite.

Under the additional hypothesis, the dimension formula gives $a+b=\dim\spn A+\dim\spn B=n$. The vectors $u_1,\dots,u_a,v_1,\dots,v_b$ are linearly independent: a linear relation would produce a vector lying in both $\spn A$ and $\spn B$, which must vanish, and then both partial relations are trivial. Let $D$ be the square matrix with these columns, so that $A+B=D(\Z^n)$. Write the Smith normal form $D=P\,\mathrm{diag}(d_1,\dots,d_n)\,Q$ with $P,Q\in GL_n(\Z)$. Then $D(\Z^n)=P(d_1\Z\oplus\dots\oplus d_n\Z)$, and since $P$ is an automorphism of $\Z^n$, the index is $|d_1\cdots d_n|=|\det D|$.
\end{proof}

The next identity is classical and appears, in precisely this form, in \cite[Lemma~5.3]{BertrandBihan2013}; we give a complete proof.

\begin{lemma}\label{lem:dual}
Let $A,B\subset N$ be saturated subgroups with $\spn A+\spn B=N_\R$, and set $C=A\cap B$. Then $[N:A+B]=[C^\perp:A^\perp+B^\perp]$, where $^\perp$ denotes annihilators in $M$.
\end{lemma}

\begin{proof}
The subgroup $C$ is saturated, because it is the intersection of two saturated subgroups. Since all these subspaces are rational, $\spn C=\spn A\cap\spn B$: a rational basis of the intersection can be cleared of denominators to give lattice vectors in both saturated subgroups. Hence $N'=N/C$ is free. Put $A'=A/C$ and $B'=B/C$. These are saturated in $N'$: if $k\bar x\in A'$ with $k\ne0$, then $kx\in A+C=A$, hence $x\in A$. Moreover $A'\cap B'=(A\cap B)/C=0$, and $\rank A'+\rank B'=\dim\spn A+\dim\spn B-2\dim\spn C=n-\dim\spn C=\rank N'$, by the dimension formula and the hypothesis $\spn A+\spn B=N_\R$.

The natural map $N/(A+B)\to N'/(A'+B')$ is an isomorphism. The dual lattice $\Hom(N',\Z)$ is identified with $C^\perp\subset M$, and under this identification the annihilators of $A'$ and $B'$ are $A^\perp$ and $B^\perp$. It therefore suffices to prove that $[N':A'+B']=[\Hom(N',\Z):A'^\perp+B'^\perp]$.

Consider the homomorphism $\varphi\colon N'\to N'/A'\oplus N'/B'$ given by $\varphi(x)=(x\bmod A',x\bmod B')$. Its kernel is $A'\cap B'=0$, the target is free because $A'$ and $B'$ are saturated, and both sides have the same rank. So $\varphi$ is given by a square integer matrix with nonzero determinant.

The map $(\bar a,\bar b)\mapsto a-b$ induces a surjection from $N'/A'\oplus N'/B'$ onto $N'/(A'+B')$. Its kernel is the image of $\varphi$. Indeed, if $a-b=a'+b'$ with $a'\in A'$ and $b'\in B'$, then $z=a-a'=b+b'$ satisfies $\varphi(z)=(\bar a,\bar b)$; and conversely $\varphi(z)$ maps to $z-z=0$. Hence the cokernel of $\varphi$ is isomorphic to $N'/(A'+B')$, and its order is $|\det\varphi|$, by the Smith normal form as in Lemma~\ref{lem:index}.

The dual map $\varphi^\vee\colon\Hom(N'/A',\Z)\oplus\Hom(N'/B',\Z)\to\Hom(N',\Z)$ has image $A'^\perp+B'^\perp$, because $\Hom(N'/A',\Z)$ is identified with $A'^\perp$ and $\Hom(N'/B',\Z)$ with $B'^\perp$. Its matrix is the transpose of the matrix of $\varphi$, so its cokernel has order $|\det\varphi^\vee|=|\det\varphi|$. Comparing the two orders gives $[N':A'+B']=[\Hom(N',\Z):A'^\perp+B'^\perp]$, which proves the lemma.
\end{proof}

\end{document}